\documentclass[12pt,twoside]{article}
\usepackage{amsmath,amssymb,amsthm,euscript,color}

\usepackage{graphicx}

\numberwithin{equation}{section}
\newtheorem{theorem}{Theorem}[section]

\newtheorem{lemma}[theorem]{Lemma}
\newtheorem{proposition}[theorem]{Proposition}

\def\al{\alpha}
\def\be{\beta}
\def\ga{\gamma}
\def\de{\delta}

\def\om{\omega}

\newcommand{\R}{\mathbb{R}}

\renewcommand{\Re} {\mathop{\mathrm{Re}}}

\begin{document}

\title{Oscillation and Instability in Chemical Reactions}
\author{Jinghua Yao$^{\dagger}$, Xiaoyan Wang$^{\ddagger}$}
\date{October 1, 2015}
\maketitle

\begin{abstract}
We prove that the famous diffusive Brusselator model can support more complicated
spatial-temporal wave structure than the usual temporal-oscillation from a standard Hopf bifurcation.  In our investigation, we discover that the diffusion term in the model is neither a usual parabolic stabilizer nor a destabilizer as in the Turing instability of uniform state, but rather plays the role of maintaining an equivariant Hopf bifurcation spectral mechanism. At the same time, we show that such a mechanism can occur around any nonzero wave number and this finding is also different from the former works where oscillations caused by diffusion can cause the growth of wave structure only at a particular wavelength. Our analysis also demonstrates that the complicated spatial-temporal oscillation is not solely driven by the inhomogeneity of the reactants.
\end{abstract}

\thanks{\textbf{Keywords}}: chemical reaction oscillation, chemical reaction instability, symmetry, spectrum, sectorial operator.\\


\section{Introduction}

We study in this paper the oscillations and instabilities in reaction-diffusion chemical systems from uniform states by rigorous mathematical analysis. 

The investigation of reaction-diffusion system that forms a pattern from uniform state dates back to the famous 1952 paper of A. Turing \cite{Tu}. In this far-reaching work \cite{Tu}, Turing did careful linear stability analysis for a reaction-diffusion system with two interacting chemicals. The analysis in \cite{Tu} had led to several important insights. One of the most surprising insight is that diffusion in a reacting chemical system can actually be a destabilizing factor that leads to instability. This is contrary to 
the intuition that diffusion smooths out spatial variations of a concentration field. A second insight is that the instability caused by diffusion can cause the growth of wave structure at a particular wavelength (or equivalently at a particular wave number). Almost at the same time as Turing, B. P. Belousov was observing oscillating chemical reactions in his laboratory. However, around the middle of last century, the chemical oscillations were thought to be inconsistent with the commonly accepted fact that a closed system of mixed chemicals must relax to equilibrium monotonically. Therefore, Belousov's work was not quickly appreciated. It was only until A. M. Zhabotinsky's systematic studies on spatially uniform oscillations that people began to realize the possibility of the chemical oscillations. During the period from late 60s to mid 70s, the works of G. Nicolis, R. Lefever, I. Prigogine and their coworkers (see \cite{NP, PL1, PL2, HN, LNH, Le} and the references therein) greatly enhanced our understanding of instability and oscillation phenomena in purely dissipative system involving chemical reactions and diffusions. Since then, the studies of chemical oscillation and instabilities have been attracting more and more attention till today in both 
the chemical physics community and mathematical community.

Motivated by the above mentioned developments, we study the so-called Brusselator model in chemical reaction and show that this model supports more complicated oscillation wave structures than those previously identified. More precisely, we are able to show that there are at least two families of standing waves and a torus of standing waves bifurcated from a uniform state for the Brusselator model in our setting. We also discover that the diffusion in our study plays the role of maintaining an $O(2)$-Hopf bifurcation mechanism and it is not a destabilizer in the sense of Turing. Meanwhile, the bifurcated wave structure can happen at any nonzero wave number. Our analysis puts the existence of oscillation waves in the chemical reaction systems on a solid theoretical foundation, complementing the former studies in \cite{Tu, PL2} which are essentially linear stability analysis from mathematical point of view.

Now we describe the model we investigate, i.e., the Brusselator model. The Brusselator model, also known as the trimolecular model, is a very famous model
for the study of cooperative processes in chemical kinetics. It is associated with the following chemical reaction system \cite{NP, PL1, PL2, Ty, Le}

\begin{equation}
\label{1a}
B_{in}\rightarrow X
\end{equation}
\begin{equation}
\label{1b}
A_{in}+X\rightarrow Y + D
\end{equation}
\begin{equation}
\label{1c}
2X+Y\rightarrow 3X
\end{equation}
\begin{equation}
\label{1d}
X\rightarrow E
\end{equation}
in which $A_{in}$ and $B_{in}$ are reactants, $D$ and $E$ are products or output chemicals while $X$ and $Y$ are intermediates.

Let $[Q]$ be the concentration of a chemical $Q$. Then the partial differential equations governing the evolution process, i.e., changes of $[X]$ and $[Y]$, are given by the following mathematical system \cite{KP, Ad, PL1, PL2, Ty} in the one spatial dimensional setting

\begin{equation}\label{reaction}
\begin{cases}
\partial_t [X]=k_1[B_{in}]-k_2[A_{in}][X]+k_3[X]^2[Y]-k_4[X]+D_1\partial^2_x[X],\\
\partial_t [Y]=k_2[A_{in}][X]-k_3[X]^2[Y]+D_2\partial^2_x[Y].
\end{cases}
\end{equation}

The positive constants $k_i$ ($1\leq i\leq 4$) are reaction rates in the four steps \eqref{1a}-\eqref{1d} and the positive constants $D_1$ and $D_2$ are diffusion rates of $X$ and $Y$ respectively. The evolution of $[X]$ is caused by the joint effects of creation in \eqref{1a} corresponing to the term $k_1[B_{in}]$, annihilation in \eqref{1b} corresponding to the term
$-k_2[A_{in}][X]$, creation in \eqref{1c} corresponding to the term $k_3[X]^2[Y]$, annihilation in \eqref{1d} corresponding to the term $-k_4[X]$, and diffusion corresponding to the term $D_1\partial^2_x[X]$ respectively. This is the first equation in \eqref{reaction}. The second equation in \eqref{reaction} can be interpreted similarly.

Introducing the notations $U_1=U_1(x,t):=[X]$, $U_2=U_2(x,t):=[Y]$ and controlling concentratins $[A_{in}]$ and $[A_{in}]$ so that $[A_{in}]\equiv A$ and $[B_{in}]\equiv B$,  we obtain the following system of partial differential equations

\begin{equation}\label{reaction1}
\begin{cases}
\partial_t U_1=k_1B-k_2AU_1+k_3U_1^2U_2-k_4U_1+D_1\partial^2_xU_1,\\
\partial_t U_2=k_2AU_1-k_3U_1^2U_2+D_2\partial^2_xU_2.
\end{cases}
\end{equation}
We can perform a nondimensionalization procedure for the above system by setting
$$u_1:=\Big(\frac{k_3}{k_4}\Big)^{1/2}U_1, u_2:=\Big(\frac{k_3}{k_4}\Big)^{1/2}U_2, \alpha:=\Big(\frac{k_3}{k_4}\Big)^{1/2}\frac{k_1B}{k_4}, $$
$$\beta:=\frac{k_2A}{k_4}, \, \bar t:=k_4 t, \, \delta_1:=D_1/k_4,\, \delta_2:=D_2/k_4,$$
and obtain a system on $(u_1, u_2)^T$ which takes on the following form after dropping the bar in $\bar{t}$:

\begin{equation}\label{Brusselator0}
\begin{cases}
\partial_t u_1=\al-\be u_1+u_1^2 u_2-u_1+\delta_1\partial_x^2 u_1,\\
\partial_t u_2=\be u_1 -u_1^2u_2+\delta_2\partial_x^2 u_2.
\end{cases}
\end{equation}

We also remark that a most mathematical convenient way to obtain the form of \eqref{Brusselator0} from
\eqref{reaction1} is simply setting $k_2, k_3$ and $k_4$ as unity and making the following identifications:
$$U_1\rightarrow u_1, U_2\rightarrow u_2, \, A\rightarrow \be,\, k_1 B\rightarrow \al,\, D_1\rightarrow \de_1, \,D_2\rightarrow \de_2.$$
See also the treatment in \cite{PL2}.

In the current work, we rigorously prove that the chemical mechanism \eqref{1a}-\eqref{1d}
supports more complicated oscillations than those oscillations identified in the previous works by showing here that
 \eqref{reaction1} bifurcates both standing waves and rotating waves when the parameter $\be$ varies around some specific values. Meanwhile, we verify that these spatial-temporal oscillations are not induced by the inhomogeneity of $u_1$ and $u_2$ during the chemical reaction process.

Rearranging \eqref{Brusselator0}, we obtain the following system
\begin{equation}
\label{Brusselator}
\begin{cases}
\partial_t u_1=\delta_1\partial_x^2 u_1 -(\be+1)u_1+u_1^2 u_2+\al,\\
\partial_t u_2=\delta_2\partial_x^2 u_2+\be u_1 -u_1^2u_2,
\end{cases}
\end{equation}
in which $\de_1,\de_2, \al$ and $\be$ are positive parameters and $u=(u_1, u_2)^T$ is a function of space and time.

We will assume $x\in [-\pi, \pi]$ and assume the state variable $u$ satisfies periodic boundary condition. Here the choice of a spatial length $2\pi$ is nonessential. Actually $\pi$ can be replaced by any positive number. For the choice of periodic boundary condition, we follow \cite{GG} which is a natural boundary condition (see Page 100 of \cite{GG}). In other words, we consider the system \eqref{Brusselator} in the spatial domain $\mathbb T:=\R/[-\pi, \pi]$.

Assume $\be_1=1+\al^2+\de_1+\de_2$ and $\mu=\be-\be_1$. Our main result is the following theorem. 

\begin{theorem}\label{thm} 
System \eqref{Brusselator} can maintain an $O(2)$-Hopf bifurcation around the uniform state $(\al, \frac{\be_1}{\al})$ in the Hilbert space $H^1(\mathbb T)\times H^1(\mathbb T)$ when $\be$ varies around $\be_1$. There are two families of bifurcated rotating waves and a torus of bifurcated standing waves. The above wave structure can occur around any nonzero wave number. Moreover, the mean-zero perturbations of the state variable from the corresponding uniform state do not support such an oscillation wave structure.
\end{theorem}

Let $\bar z$, $\mbox{Re}\, z$ and $\mbox{Im}\, z$ be the complex conjugate, real part and imaginary part of a complex number $z$ respectively. The following theorems explain the dynamics of the system \eqref{Brusselator} given by Theorem \ref{thm} more precisely.

\begin{theorem}
System \eqref{Brusselator} admits a center manifold reduction with $O(2)$ symmetry near $\mu=0$.  If the center space is parametrized by 
$z_1\xi_1 +z_2\xi_2 +\bar z_1 \bar \xi_1 +\bar z_2\xi_2$ with $z_1,z_2\in\mathbb C^1$
and
\begin{equation}
\om=\sqrt{\al^2(1+\de_1-\de_2)-\de_2^2},\, \xi_1=e^{ix}\begin{pmatrix}1\\\frac{-\al^2-\de_2+i\omega}{\al^2} \end{pmatrix},\, \xi_2=e^{-ix}\begin{pmatrix}1\\\frac{-\al^2-\de_2+i\omega}{\al^2} \end{pmatrix},
\end{equation}
then the dynamics on the center manifolds has the following form
\begin{equation}\label{O2 result}
\begin{cases}
\frac{d}{dt}z_1=i\omega z_1+ z_1 P(|z_1|^2, |z_2|^2,\mu)+\zeta(z_1, z_2, \bar z_1, \bar z_2, \mu),\\
\frac{d}{dt}z_2=i\omega z_1+ z_2 P(|z_2|^2, |z_1|^2,\mu)+\zeta(z_2, z_1, \bar z_2, \bar z_1, \mu),
\end{cases}
\end{equation}
in which $P$ is a polynomial of degree $p\in \mathbb N$ in its first two arguments with coefficients
depending on $\mu$, and has the form $P_{\mu}(|z_1|^2, |z_2|^2)=a\mu+b|z_1|^2+c|z_2|^2+h.o.t$
where $a$, $b$, $c$ are given by Proposition \ref{proposition abc} and $\zeta(z_1, z_2, \bar z_1, \bar z_2, \mu)=O((|z_1|+|z_2|)^{2p+3})$ satisfies
$$\zeta(e^{i\phi}z_1, e^{i\phi}z_2, e^{-i\phi}\bar z_1,e^{-i\phi} \bar z_2, \mu)=e^{i\phi}\zeta(z_1, z_2, \bar z_1, \bar z_2, \mu).$$
\end{theorem}

To analyze the equivariant Hopf bifurcation dynamics of \eqref{O2 result}, it is sufficient to consider the third order truncated system. The truncated system of \eqref{O2 result} at order three has the following form
\begin{equation}\label{normal}
\begin{cases}
\frac{d}{dt}z_1=i\omega z_1+ z_1(a \mu+b |z_1|^2+c |z_2|^2),\\
\frac{d}{dt}z_2=i\omega z_2+z_2(a\mu+b|z_2|^2+c|z_1|^2),
\end{cases}
\end{equation}
which can be written as follows
\begin{equation}\label{normalMK}
\begin{cases}
\frac{d}{dt}z_1=(a\mu+i\omega) z_1+ cz_1( |z_1|^2+ |z_2|^2)+(b-c)|z_1^2|z_1,\\
\frac{d}{dt}z_2=(a\mu+i\omega) z_2+ cz_2( |z_1|^2+ |z_2|^2)+(b-c)|z_2^2|z_2.
\end{cases}
\end{equation}
The latter, i.e., \eqref{normalMK}, is in accordance with that of \cite{MK} by the following correspondence
\begin{equation}\label{correspondance}
a\mu\rightarrow\lambda,\, c\rightarrow A,\, b-c\rightarrow B.
\end{equation}
A sufficient condition for \eqref {normalMK} to support both bifurcating standing waves and bifurcating rotating waves (see \cite{MK})
consists of the following relations
\begin{equation}\label{S1}
\mbox{Re}\, B\neq 0, \, \mbox{Re}\, A+\mbox{Re}\, B\neq 0,\, 2\mbox{Re}\, A+\mbox{Re}\, B\neq 0.
\end{equation}
Concerning the stability of these bifurcated waves, we have the following schematic graph from \cite{MK}:
\begin{center}
\includegraphics[scale=0.5]{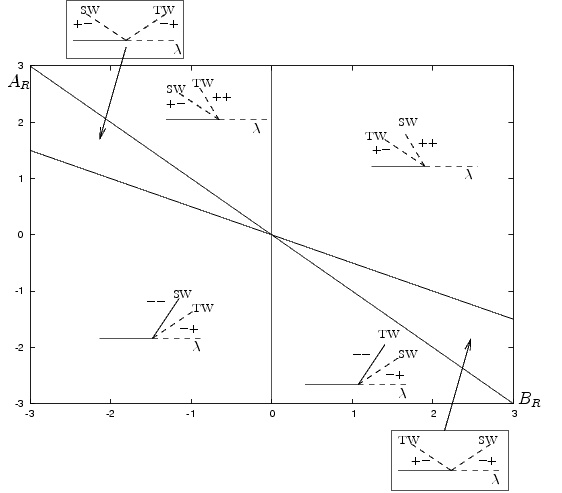}
\end{center}
In the graph, ``$A_R$" and ``$B_R$" means the real parts of $A$ and $B$ respectively. ``TW" represents rotating wave and ``SW" represents standing wave. The two degeneracy (in the sense of \eqref{S1}) lines in the box are given by the equations $\mbox{Re}\, A+\mbox{Re}\, B=0$ and $2\mbox{Re}\, A+\mbox{Re}\, B=0$. The solid (resp., dashed) lines indicate (orbitally) stable (resp., orbitally unstable) of bifurcated solutions. 

Rephrasing these relations in terms of $a, b, c$ for \eqref{normal}, the sufficient condition becomes
\begin{equation}
\Re\, b\not=0,\, \Re\, (b\pm c)\not=0.
\end{equation}
Though the above condition is by no means necessary for the bifurcated nontrivial solutions, we can verify it for
some (hence uncountable many) specific parameters $\de_1$, $\de_2$ and $\al$.
For our purpose, it suffices to show that there are parameter values such that the above condition holds.
Following \cite{PL2}, we choose $\de_1=\de_2=1$, $\al=2$. Then we have $\be_1=7$ and $\omega=\sqrt{3}$ by
\eqref{beta1} and \eqref{omega}. With these set of datum, we obtain
$B_r=18$, $B_i=-33$, $N_r=66$, $N_i=12$, $\mbox{Re}\, b=-\frac{17}{8}$.
We also get $C_{2r}=-192$, $C_{2i}=72\sqrt{3}$, $P_2(0)=12$, $Q_r=42$,
$Q_i=-20\sqrt{3}$, $\mbox{Re}\, c=\frac{11}{4}$. For expressions of the above quantities, see
Proposition \ref{proposition abc}. With the above choice of $\de_1, \de_2$ and $\al$, we have
$\mbox{Re}\, b<0$ and $\mbox{Re}\, b+\mbox{Re}\, c=\frac{5}{8}>0$. In view of
\eqref{correspondance}, these two relations correspond
to the situation $\mbox{Re}\, A+\mbox{Re}\, B<0$ and $2\mbox{Re}\, A+\mbox{Re}\, B>0$.

\textit{Discussion and Problems.} Let us first comment on the reaction mechanism in the model \eqref{1a}-\eqref{1d}. The trimolecular reaction step, i.e., \eqref{1c}, arises in the formation of ozone by atomic oxygen via a triple collision. It also arises in enzymatic reactions and in plasma and laser physics \cite{Ad, KP, TGC} etc. It has been known \cite{Ad, NP, PL1} since late sixties of last century that the problem of dealing with chemical reactions of systems involving two variable intermediates (in the Brusselator model, $X$ and $Y$) together with certain reactants whose concentrations are assumed to be controlled throughout the reaction process is both significant in the investigation of chemical reaction mechanisms and pragmatic in applications.  As is well-known, a state of homogeneity and equilibrium is quickly reached for most of the chemical reactions. While for the Brusselator model, a distinct and surprising feature is that it supports very complicated oscillations and produces chemical reaction instabilities. Our results are in accordance with the results in
the celebrated works \cite{PL1, PL2} on chemical reaction oscillations and instabilities. As we always
have $\mbox{Re}\, a\neq0$ here, we only need $\mbox{Re}\, b\neq0$ or $\mbox{Re}\, c\neq0$,
or $\mbox{Re}\, b+\mbox{Re}\, c\neq0$ to have spatial-temporal oscillational dynamics. Therefore, we have demonstrated more complicated mechanism for chemical reaction oscillations and instabilities in the current
study. Motivated by our current study, a very intriguing problem would be rigorously analyzing the mechanisms of oscillations and instabilities for the reaction schemes proposed in \cite{PL2} at the nonlinear level. A related question is the construction of reaction schemes which can exhibit equivariant oscillations or instabilities with respect to given isometry groups. In other words, it would be interesting to further investigate the roles of symmetry in chemical reactions.\\

\textbf{Acknowledgments.} J. Yao would like to thank Professor Vladimir Sverak of University of Minnesota (UMN) for his suggestion to add a graph for the dynamics in the current work during Yao's visit to
UMN in April 2015. He also thanks Professor Xiangbing Meng of Obstetrics and Gynecology in University of Iowa for explaining the applications of chemical reaction instabilities from biomedical point of view. The research of X. Wang was
partially supported by ONR grant N00014-15-1-2385, NSF grants DMS-1206438, DMS-1510249, and by the Research
Fund of Indiana University.


\section{Proof of Main Result}

To show our main result, we will carry out a complete equivariant bifurcation analysis. 
We divide this part into subsections to make the argument structure concrete.

\subsection{Onset of bifurcation}\label{Onset of bifurcation}
Our goal in this section is to identify the onset of possible bifurcations. The following analysis is essentially
a linear stability analysis.

Introducing the new variable $v=(v_1, v_2)^T:=(u_1-\al, u_2-\frac{\be}{\al})^T$,
system \eqref{reaction1} can be written in terms of $v$ as

\begin{equation}
\label{perturbation system}
\begin{cases}
\partial_t v_1=\delta_1\partial_x^2 v_1 +(\be-1)v_1+\al^2 v_2+ 2\al v_1 v_2+\frac{\be}{\al}v_1^2+v_1^2 v_2,\\
\partial_t v_2=\delta_2\partial_x^2 v_2-\be v_1 - \al^2 v_2 - 2\al v_1 v_2-\frac{\be}{\al}v_1^2-v_1^2v_2,
\end{cases}
\end{equation}
which can be identified as an operator equation
\begin{equation}
\label{operator equation}
\partial_t v=\bold L_{\be}v+\bold R(v)
\end{equation}
with
$$\bold L_{\be}=\begin{pmatrix} \delta_1\partial^2_x +\be-1 & \al^2\\ -\be & \delta_2\partial^2_x -\al^2\end{pmatrix},$$
$$\bold R(v)=\begin{pmatrix} 2\alpha v_1 v_2 + \frac{\be}{\al}v_1^2+v_1^2 v_2 \\-2\al v_1 v_2 - \frac{\be}{\al}v_1^2-v_1^2 v_2\end{pmatrix}.$$

{\it The Euclidean symmetry}. The equation \eqref{operator equation} is equivariant under the Euclidean group $O(2)$, which can be easily observed due to the form of the equation and the periodic boundary condition. Indeed, define
the group actions $R(\phi)$ and $S$ where $\phi\in \mathbb R^1/2\pi\mathbb Z$ as follows
$$R(\phi)v(x)=v(x+\phi),\, Sv(x)=v(-x).$$
It is obvious that $S^2=\mbox{Id}$. As only second order derivatives and constants are involved in the operator $\mathbf L_{\be}$, we have $[\mathbf L_{\be}, S]=0$.
Also, it is easy to verify that $SR(\phi)=R(-\phi)S$. As there are no derivatives involved in the nonlinear term $\mathbf R(v)$, we also have
$[\mathbf R, S]=0$. We can also easily observe that $[R(\phi), \mathbf L]=0$ and $[R(\phi), \mathbf R]=0$. Therefore, \eqref{operator equation} is equivariant with respect to the Euclidean group $O(2)$.

Now we regard $\mathbf L_{\be}$ as a linear operator on the space $L^2(\mathbb T)\times L^2(\mathbb T)$ with domain 
$H^2(\mathbb T)\times H^2(\mathbb T)$, i.e.,
$$\mathbf L_{\be}: H^2(\mathbb T)\times H^2(\mathbb T) \subset L^2(\mathbb T)\times L^2(\mathbb T)\rightarrow L^2(\mathbb T)\times L^2(\mathbb T).$$
To identify the onset of bifurcation, we study the spectral equation $\mathbf L_{\be}v=\lambda v$ for $\lambda\in\mathbb C^1$.

By Fourier analysis, the spectra of $\mathbf L_{\be}$ are given by the eigenvalues of the
matrices $M_n$ for $n\in\mathbb Z$ where 
$$M_n=\begin{pmatrix}-n^2\de_1+\be-1 & \al^2\\ -\be & -n^2\de_2-\al^2 \end{pmatrix}.$$
This corresponds to seek solutions $v=(v_1, v_2)$ of the form 
$v_1=\sum_{n\in \mathbb Z}v_1^{(n)}e^{in\frac{2\pi}{2\pi}x}=\sum_{n\in \mathbb Z}v_1^{(n)}e^{inx}$ and 
$v_2=\sum_{n\in \mathbb Z}v_2^{(n)}e^{in\frac{2\pi}{2\pi}x}=\sum_{n\in \mathbb Z}v_2^{(n)}e^{inx}$ where $v_1^{(n)}$ and $v_2^{(n)}$ are complex numbers. Therefore, the
matrices $M_n$ have been indexed by wave numbers $n$.

The characteristic equations of $M_n$ are 
\begin{equation}\label{Pn}
P_n(\lambda,\be)=\lambda^2+(\be(n)-\be)\lambda+\ga(n)-n^2\de_2\be
\end{equation}
with
$$\be(n)=1+\al^2+n^2(\de_1+\de_2),$$
$$\ga(n)=n^2\de_2+n^2\de_1\al^2+n^4\de_1\de_2+\al^2.$$

We consider the dynamics of the system \eqref{Brusselator} when $\be$ varies around $\be(1)$. For convenience, we will use the following notation
\begin{equation}\label{beta1}
\be_1:=\be(1)=1+\al^2+\de_1 +\de_2.
\end{equation}

When the wave number $n=0$, we have
$$P_0(\lambda, \be_1)=\lambda^2+(1+\al^2-\beta_1)\lambda+\al^2=\lambda^2-(\de_1+\de_2)\lambda+\al^2.$$
Therefore, $P_0(\lambda, \be_1)$ has two roots with positive real parts (either a pair of complex conjugate roots or two positive real roots).

Before we consider the contribution of nonzero wave numbers $n$, we remark the spectral analysis around the zero wave number has the following important consequence:

\begin{proposition}
There is no Turing instability for the uniform state $(\alpha, \frac{\beta_1}{\alpha})$, i.e., the diffusion terms are not a destabilizer in the sense of Turing.
\end{proposition}

\begin{proof}
By the definition of Turing instability for a uniform state, a Turing unstable uniform state should be spectrally stable without diffusion and becomes unstable when diffusion effects are taken into account. If there were no diffusion terms, the stability of uniform state here is determined by the spectra of $\mathbf L_{\be_1}$ contributed by the wave number $n=0$. However, $P_0(\lambda, \be_1)$ has two roots with positive real parts. Therefore, the conclusion of the proposition follows.
\end{proof}

When the wave number $n\not=0$, we first notice that $\be(\pm n)=\be(|n|)$ and the latter is strictly increasing as a function of $|n|\in \mathbb N$. By elementary inequality, we know that 
\begin{align*}
\ga(n)-n^2\de_2\be&=n^2\de_2(1+\al^2\frac{\de_1}{\de_2}+n^2\de_1+\frac{\al^2}{n^2\de_2}-\be)\\
&\geq n^2\de_2(1+\al^2\frac{\de_1}{\de_2}+2\sqrt{n^2\de_1\frac{\al^2}{n^2\de_2}}-\be)\\
&= n^2\de_2\Big((1+\al\sqrt{\de_1/\de_2})^2-\be\Big).\\
\end{align*}
Notice that the quantity in the bracket of last line does not depend on the wave number $n$.

In the following, we assume 
\begin{equation}\label{assumption}
\be_1<(1+\al\sqrt{\de_1/\de_2})^2.
\end{equation}
Then we know $\ga(n)-n^2\de_2\be_1>0$ uniformly in $n\in \mathbb Z$. In particular, when $n=\pm 1$, we have
$$\de_2(1+\al^2\frac{\de_1}{\de_2}+\de_1+\frac{\al^2}{\de_2}-\be_1)=\al^2(1+\de_1-\de_2)-\de_2^2>0.$$

When $\be=\be_1$, both $P_{1}(\lambda,\be_1)$ and $P_{-1}(\lambda,\be_1)$ have the same pair of conjugate
pure imaginary roots $\pm i\omega$ with $\omega>0$ given below
\begin{equation}\label{omega}
\omega^2=\al^2(1+\de_1-\de_2)-\de_2^2.
\end{equation}

When $\be=\be_1$, $P_n(\lambda,\be_1)$ have no roots lying on the imaginary axis in the complex plane for any wave number $n$ with $|n|\geq 2$. This can be easily observed by checking the coefficients of $\lambda$ in the polynomials $P_n(\lambda, \be_1)$. 

For convenience, we will denote $P_n(\lambda,\be_1)$ by $P_n(\lambda)$ from now on. To sum up the analysis above, we have identified the onset of an $O(2)$-Hopf bifurcation. More precisely, the operator $\mathbf L_{\be_1}$ has a pair of pure imaginary eigenvalues $\pm i\om$ when $\be=\be_1$ under the assumption \eqref{assumption}. Due to the symmetry of the system \eqref{Brusselator}, both eigenvalues double and have geometric multiplicity two.

We are in a position to explain the role of diffusion terms for our study now: when applied to a single wave $e^{inx}$ with wave number $n$, the second order operator $\partial_x^2$ acts as multiplication by $-n^2$ on the wave, i.e., $-n^2e^{inx}$; while for a nonzero wave number $n$, we have $n\neq -n$, hence $e^{inx}$ and $e^{-inx}$ produce independent single waves. Therefore, it is the diffusion in system \eqref{Brusselator} that makes the $O(2)$-Hopf bifurcation spectral scenario possible. It is important to notice this mechanism of diffusion here.

Examining the above analysis, we find that the $O(2)$-Hopf bifurcation mechanism can be maintained around any nonzero wave number $n$. This is also in contrast to the situation in the classical work of Turing where the instability caused by diffusion can cause the growth of wave structure at a particular wavelength (or equivalently at a particular wave number), i.e., the second insight mentioned in the introduction.

\subsection{Spectral Analysis}

In this part, we will compute the center space of the operator $\mathbf L_{\be_1}$ and prove a resolvent estimate which we will need to show the existence of center manifolds later.

Now we proceed to compute the center space of $\mathbf L_{\be_1}$ in the space $H^2(\mathbb T)\times H^2(\mathbb T)$ or equivalently $L^2(\mathbb T)\times L^2(\mathbb T)$ because of the finite dimensionality of the center space. Due to group action and conjugacy, this is an easy task. Consider the case when wave number $n=1$. We shall seek solutions of the form $\xi_1=e^{ix}V$ for the eigenvalue $i\omega$ where
$V\in\mathbb C^2$ is a complex vector. The equation $\mathbf L_{\be_1}\xi_1=i\om \xi_1$ reduces to an algebraic equation on the complex vector $V$ given below
\begin{equation}
\begin{pmatrix}-\de_1+\be_1-1-i\om & \al^2\\ -\be_1 & -\de_2-\al^2-i\om\end{pmatrix}V=\begin{pmatrix}0\\0\end{pmatrix}.
\end{equation}
As $\det \begin{pmatrix}-\de_1+\be_1-1-i\om & \al^2\\ -\be_1 & -\de_2-\al^2-i\om\end{pmatrix}=P_1(i\om)=0$, we can choose $V$ as follows

$$V=\begin{pmatrix}1\\\frac{-\al^2-\de_2+i\omega}{\al^2} \end{pmatrix}.$$
The corresponding eigenfunction to $\bold L_{\be_1}$ is
$$\xi_1=e^{ix}\begin{pmatrix}1\\\frac{-\al^2-\de_2+i\omega}{\al^2} \end{pmatrix}.$$


As $[\mathbf L_{\be_1}, S]=0$, we know 

$$\xi_2:=S\xi_1=e^{-ix}\begin{pmatrix}1\\\frac{-\al^2-\de_2+i\omega}{\al^2} \end{pmatrix}$$
is also an eigenfunction of $\mathbf L_{\be_1}$ corresponding to the eigenvalue $i\om$. Indeed, we have
$$\mathbf L_{\be_1}\xi_2=\mathbf L_{\be_1}S\xi_1=S\mathbf L_{\be_1}\xi_1=S(i\om \xi_1)=i\om S\xi_1=i\om\xi_2.$$

By conjugacy, we know that
$$\bar{\xi}_1=e^{-ix}\begin{pmatrix}1\\\frac{-\al^2-\de_2-i\omega}{\al^2} \end{pmatrix},\,\,\bar{\xi_2}=e^{ix}\begin{pmatrix}1\\\frac{-\al^2-\de_2-i\omega}{\al^2} \end{pmatrix}$$ 
are also eigenfunctions of $\mathbf L_{\be_1}$ and they are associated with the eigenvalue $-i\om$.

The indices will be important for the computations in determining the parameters in the bifurcation dynamics. Here $\xi_1$ and $\xi_2$ are eigenfunctions associated with $i\om$. Hence $\bar{\xi_1}$ and $\bar{\xi_2}$ are eigenfunctions associated with $-i\om$. Consequently, $\xi_1$ and $\bar{\xi_2}$ are eigenfunctions associated with the Fourier mode $n=1$ while their conjugates are associated with the Fourier mode $n=-1$.

As we have identified the center space of $\mathbf L_{\be_1}$, we parametrize the center space of $\mathbf L_{\be_1}$ by using complex conjugate coordinates as follows
\begin{equation}\label{center space}
\{z_1\xi_1 +z_2\xi_2 +\bar z_1 \bar \xi_1 +\bar z_2\xi_2\, |\, z_1,z_2\in\mathbb C^1\}.
\end{equation}

For later use, here we also compute a normalized dual eigenfunction of $\xi_1$, i.e., an eigenfunction $\xi_1^{*}$ which is in $\ker (i\omega-\bold L_{\be_1})^{*}$ such that $\langle \xi_1^{*}, \xi_1  \rangle=1$. Noticing the form of $i\om-\bold L_{\be_1}$, i.e., 
$$i\om-\bold L_{\be_1}=i\om \mbox{Id}-\begin{pmatrix} \delta_1\partial^2_x +\be_1-1 & \al^2\\ -\be_1 & \delta_2\partial^2_x -\al^2\end{pmatrix},$$
we shall seek a solution of the form $e^{ix}W$ with the constant vector $W\in\mathbb C^2$. Then the
action of $(i\omega-\bold L_{\be_1})^*$ on such a solution is equivalent to the action of the following matrix to the vector $W$:
$$-i\om I_2-\begin{pmatrix} -\delta_1 +\be_1-1 & -\be_1\\ \al^2 & -\delta_2 -\al^2\end{pmatrix}=-\begin{pmatrix} -\delta_1 +\be_1-1+i\om & -\be_1\\ \al^2 & -\delta_2 -\al^2+i\om\end{pmatrix}.$$
Therefore, $W$ shall parallel to the following vectors
$$\begin{pmatrix} \frac{\de_2+\al^2-i\om}{\al^2} \\ 1\end{pmatrix},\,\,\begin{pmatrix} 1\\ \frac{\de_2+\al^2+i\om}{\be_1}\end{pmatrix}.$$
Notice that the above two vectors are parallel. In view that the second component of $\xi_1$ is more complicated than its first component 1, we could pick the dual eigenfunction $\xi_1^*$ according to the first vector as follows
\begin{equation}\label{xi star}
\xi_1^{*}=i\frac{\al^2}{4\pi\om}e^{ix}\begin{pmatrix} \frac{\de_2+\al^2-i\om}{\al^2} \\ 1\end{pmatrix}.
\end{equation}
With the above choice of $\xi_1^*$, it is easy to check that $\langle \xi_1^{*}, \xi_1  \rangle=1$.


Next, we prove a lemma concerning the behavior of the resolvent of the operator $\mathbf L_{\be_1}$.

\begin{lemma}\label{Resolvent estimate}
 There exists a constant $\om_0>0$, such that $i\lambda\in \rho(\mathbf L_{\be_1})$ for any real number $\lambda$ with
$|\lambda|>\om_0$ and the following resolvent estimate holds:
$$\|(i\lambda-\mathbf L_{\be_1})^{-1}\|_{L^2(\mathbb T)\times L^2(\mathbb T)\rightarrow L^2(\mathbb T)\times L^2(\mathbb T)}\leq \frac{M}{|\lambda|}$$ for some fixed positive constant $M$.
\end{lemma}

\begin{proof} First, we claim that the operator
$\mathbf A: H^2(\mathbb T)\subset L^2(\mathbb T)\mapsto L^2(\mathbb T)$	defined by $\mathbf A v=-\de_1\partial_x^2v$ is a sectorial operator. Indeed, the spectral set $\sigma(\mathbf A)$ of the operator $A$ is $\{\de_1 n^2\,| n\in\mathbb Z\}$. In particular, the set $S_{-\de_1, \frac{\pi}{4}}:=\{z\in\mathbb C^1\, | \frac{\pi}{4}\leq |\, arg(z-(-\de_1))|\leq \pi, z\not=-\de_1\}$ is a subset of $\rho(\mathbf A)$. For any $z\in S_{-\de_1, \phi}$, an easy geometic argument yields $|z-\de_1 n^2|\geq (\sin\frac{\pi}{4})|z-(-\de_1)|$ for any $n\in\mathbb Z$. For these $z$, we consider the operator equation $(z-\mathbf A)v=w$. Assume $w=\sum_{n\in\mathbb Z}w(n)e^{inx}$ where $w(n)$'s are the Fourier coefficients of $w$. Then $v$ can be solved explicitly as $v=\sum_{n\in\mathbb Z}v(n)e^{inx}$ with $v(n)=\frac{v(n)}{\lambda-\de_1 n^2}$. By Plancherel Theorem, we obtain 
\begin{align*}
|v|^2_{L^2(\mathbb T)}&=\sum_{n\in\mathbb Z}|v(n)|^2=
\sum_{n\in\mathbb Z}\Big|\frac{w(n)}{\lambda-\de_1 n^2}\Big|^2\\
&\leq\sum_{n\in\mathbb Z}
\frac{1}{(\sin^2\frac{\pi}{4})|z-(-\de_1)|^2}|w(n)|^2\leq \frac{2}{|z-(-\de_1)|^2}|w|^2_{L^2(\mathbb T)},
\end{align*}
which implies the resolvent estimate
$$\|(z-\mathbf A)^{-1}\|_{L^2(\mathbb T)\rightarrow L^2(\mathbb T)}\leq \frac{\sqrt{2}}{|z+\de_1|}.$$
Therefore the claim is true.

Similarly, the operator $\mathbf B: H^2(\mathbb T)\subset L^2(\mathbb T)\mapsto L^2(\mathbb T)$	defined by $\mathbf B v=-\de_2\partial_x^2v$ is also a sectorial operator. Therefore, the operator $\mbox{diag}\{\mathbf A, \mathbf B\}$ is a sectorial operator on $L^2(\mathbb T)\times L^2(\mathbb T)$. By standard perturbation theory of linear operators (see for example p. 19 of \cite{He}), we conclude $-\mathbf L_{\be_1}$ is a sectorial operator. Then there exists a sector $S_{a_0,\phi}$ similary defined as $S_{-\de_1,\frac{\pi}{4}}$ with $a_0<0$ and $0<\phi<\frac{\pi}{2}$ such that
$$\|(z+\mathbf L_{\be_1})^{-1}\|\leq \frac{M}{|z-a_0|},\,\, \mbox{for any}\,z\in S_{a_0,\phi}\,\, \mbox{and some fixed postive constant}\, M.$$
Then for any real $\lambda$ with $|\lambda|\geq \om_0:= 2a_0\tan\phi$, we have $i\lambda\in \rho(\mathbf L_{\be_1})$ and the resolvent estimate
$$
\|(i\lambda+\mathbf L_{\be_1})^{-1}\|\leq \frac{M}{|-i\lambda-a_0|}\leq \frac{ M}{|\lambda|}.
$$
This completes the proof of Lemma \ref{Resolvent estimate}.
\end{proof}

We emphasize that the boundary condition involved in Lemma \ref{Resolvent estimate} is the periodic condition rather than a Dirichlet boundary condition. Our geometric proof demonstrates that the above far field resolvent estimate along the imaginary axis is a robust structural property of the operator $\mathbf L_{\be_1}$.

Next, we will isolate the bifurcation parameter, write the nonlinearity in terms of multilinear maps and show the existence of center manifolds for the system \eqref{Brusselator}. The isolation of the bifurcation parameter $\mu=\be-\be_1$ will make the linear operator in \eqref{bifurcation} below have no dependence on $\mu$ while writing the nonlinearity in terms of multilinear maps will bring us computational convenience later. Notice now $v=(u_1,u_2)^T-(\al,\frac{\be_1}{\al})^T$.

For the above purposes, we write the opeator equation \eqref{operator equation} as follows
\begin{equation}\label{bifurcation}
\partial_t v= \mathbf{L}_{\be_1}v+( \mathbf{L}_{\be}v-\mathbf{L}_{\be_1}v)+\mathbf{R}(v).
\end{equation}
Introducing the notation $\mathbf{R}(v,\mu):=(\mathbf{L}_{\be}v-\mathbf{L}_{\be_1}v)+\mathbf{R}(v)$, we write 
$\mathbf{R}(v,\mu)$ in terms of multilinear maps as the following sum
$$\mathbf{R}(v,\mu)=\mu \mathbf R_{01}v+\mathbf R_{20}(v, v)+
\mathbf R_{30}(v, v, v)+\mu \mathbf R_{21}(v,v)$$ where

$$\mathbf R_{01}v=\begin{pmatrix}v_1\\-v_1 \end{pmatrix},\,\,\mathbf R_{20}(u, v)=\begin{pmatrix}\al(u_1 v_2+ u_2 v_1)+\frac{\be_1}{\al}u_1 v_1\\ -\al(u_1 v_2+ u_2 v_1)-\frac{\be_1}{\al}u_1 v_1\end{pmatrix},$$
$$\mathbf R_{30}(u, v, w)=\frac{1}{3}\begin{pmatrix}u_1 v_1 w_2+ u_1 v_2 w_1+ u_2 v_1 w_1\\-u_1 v_1 w_2-u_1 v_2 w_1- u_2 v_1 w_1 \end{pmatrix}, \,\,\mathbf R_{21}(u,v)=\begin{pmatrix}\frac{1}{\al}u_1 v_1\\-\frac{1}{\al}u_1 v_1\end{pmatrix}.$$ 

Now, we conclude the existence of center manifolds for the system \eqref{bifurcation}. For our specific purpose in the current study, we make the following choice of Banach spaces

$$\mathbf Z=H^2(\mathbb T)\times H^2(\mathbb T) ,\,\mathbf Y=H^1(\mathbb T)\times H^1(\mathbb T),\,\,\mathbf X=L^2(\mathbb T)\times L^2(\mathbb T).$$
Then we have the following theorem.

\begin{theorem}\label{cmt}
System \eqref{bifurcation} (or equivalently system \eqref{Brusselator}) admits a parameter-dependent center manifold $\mathcal M_0(\mu)$ given through the reduction function $\Psi\in C^k(\mathbf Z_0, \mathbf Z_h)$ by $\mathcal M_0 (\mu)=\{ v_0+\Psi(v_0,\mu); v_0\in \mathbf Z_c \}$ in a neighborhood of $\mathcal O_v\times \mathcal O_{\mu}$ of $(0,0)$ and $[\Psi, R(\phi)]=0$ and $[\Psi, S]=0$.
\end{theorem}

\begin{proof}
We shall verify the assumptions in the center manifold theorem Theorem A.1 in the Appendix item by item. It is evident that $\mathbf R(0,0)=0$ and $D_v\mathbf R(0,0)=0$ from the specific form of $\mathbf R(v,\mu)$. By Sobolev embedding theorem, it is also easy to verify that $\mathbf R(v,\mu)\in \mathbf Y$. The spectral gap condition and symmetry conclusion follow from the analysis in Section \ref{Onset of bifurcation} while the resolvent estimate follows from Lemma \ref{Resolvent estimate}.
\end{proof}


\subsection{Normal form of bifurcation dynamics}

Due to Theorem \ref{cmt} and our spectral analysis, we know that the reduced dynamics of the system \eqref{bifurcation} is a four dimensional system of nonlinear ordinary differential equations. We will adopt the complex conjugate pair parametrization $z_1\xi_1+z_2\xi_2+\bar z_1\bar \xi_1+\bar z_2\bar \xi_2^*$ of the center space in \eqref{center space} throughout the article. We readily identify that the action of the operator $\mathbf L_{\be_1}$ on the center space is given by the following rules:
$$\mathbf L_{\be_1}\xi_1=i\om\xi_1,\, \mathbf L_{\be_1}\xi_2=-i\om\xi_2.$$
Due to the finite dimensionality of the center space and the spectral mapping theorem, the action of the operator $e^{t\mathbf L_{\be_1}^*}$ on the center space is given as follows
$$e^{t\mathbf L_{\be_1}^*}\xi_1=e^{i\om}\xi_1,\, e^{t\mathbf L_{\be_1}^*}\xi_2=e^{-i\om}\xi_2.$$
In view of our specific parametrization for the center space, the dynamics of \eqref{Brusselator} on the center manifolds has the following form
	\begin{equation}\label{O2}
	\begin{cases}
	\frac{d}{dt}z_1=i\omega z_1+ \mathcal N^1_{\mu}(z_1,z_2,\bar z_1, \bar z_2)+\zeta(z_1, z_2, \bar z_1, \bar z_2, \mu),\\
	\frac{d}{dt}z_2=i\omega z_2+ \mathcal N^2_{\mu}(z_1,z_2,\bar z_1, \bar z_2)+\zeta(z_2, z_1, \bar z_2, \bar z_1, \mu).
	\end{cases}
	\end{equation}
The polynomial function $(\mathcal N^1_{\mu}, \mathcal N^2_{\mu}, \overline{\mathcal N^1_{\mu}}, \overline{\mathcal N^1_{\mu}})^T$ corresponds to the function $\mathcal N_{\mu}$
in Theorem \ref{normal form theorem}. We will denote $\mathcal N_{\mu}(z_1, z_2, \bar z_1, \bar z_2):=(\mathcal N^1_{\mu}, \mathcal N^2_{\mu}, \overline{\mathcal N^1_{\mu}}, \overline{\mathcal N^1_{\mu}})^T$. Here $\zeta$ represents higher order terms.

We shall explore the function $\mathcal N_{\mu}(z_1, z_2, \bar z_1, \bar z_2)$ by Theorem \ref{normal form theorem}. This can be easily achieved by choosing specific test values of time variable $t$ and phase variable $\phi$.

In view of the action of the operator $e^{t\mathbf L_{\be_1}^*}$ on the center space, we know that the characteristic condition $\mathcal{N}_{\mu}(e^{t\mathbf{L}^*}v)=e^{t\mathbf{L}^*}\mathcal{N}_{\mu}(v)$ yields
\begin{equation}\label{1}
\mathcal N^j_{\mu}(e^{i\om t}z_1, e^{i\om t}z_2, e^{-i\om t}\bar z_1, e^{-i\om t}\bar z_2)=e^{i\om t}\mathcal N^j_{\mu}(z_1, z_2, \bar z_1, \bar z_2),\,\, j=1, 2.
\end{equation}

Noticing that $[\mathcal N_{\mu}, R(\phi)]=0$ and $[\mathcal N_{\mu}, S]=0$, we have the following conclusions:

\begin{equation}\label{2}
\mathcal N^1_{\mu}(e^{i\phi}z_1, e^{-i\phi}z_2, e^{-i\phi}\bar z_1, e^{i\phi}\bar z_2)=e^{i\phi}\mathcal N^1_{\mu}(z_1, z_2, \bar z_1, \bar z_2),
\end{equation}
\begin{equation}\label{3}
\mathcal N^2_{\mu}(e^{i\phi}z_1, e^{-i\phi}z_2, e^{-i\phi}\bar z_1, e^{i\phi}\bar z_2)=e^{-i\phi}\mathcal N^2_{\mu}(z_1, z_2, \bar z_1, \bar z_2)
\end{equation}
and
\begin{equation}\label{4}
\mathcal N^1_{\mu}(z_2, z_1, \bar z_2, \bar z_1)=\mathcal N^2_{\mu}(z_1, z_2, \bar z_1, \bar z_2).
\end{equation}

From \eqref{1} and \eqref{3}, we obtain that
\begin{equation}\label{four}
\mathcal N^1_{\mu}(e^{i(\phi+\om t)}z_1, e^{i(-\phi+\om t)}z_2, e^{-i(\phi+\om t)}\bar z_1, e^{-i(-\phi+\om t)}\bar z_2)=e^{i(\phi+\om t)}\mathcal N^1_{\mu}(z_1, z_2, \bar z_1, \bar z_2).
\end{equation}
Choosing special values of $t\in\mathbb R^1$ and $\phi\in \mathbb R^1/2\pi \mathbb Z$ such that
$$\phi+\om t=-\arg z_1,\, -\phi+\om t=-\arg z_2,$$
we have
\begin{equation}\label{5}
\mathcal N^1_{\mu}(|z_1|, |z_2|, |z_1|, |z_2|)=e^{-i\arg z_1}\mathcal N^1_{\mu}(z_1, z_2, \bar z_1, \bar z_2).
\end{equation}
Now choosing $t\in\mathbb R^1$ and $\phi\in \mathbb R^1/2\pi \mathbb Z$ such that
$$\phi+\om t=\pi,\, -\phi+\om t=0,$$
we find from \eqref{four} that
\begin{equation}\label{6}
\mathcal N^1_{\mu}(-z_1, z_2, -\bar z_1, \bar z_2)=-\mathcal N^1_{\mu}(z_1, z_2, \bar z_1, \bar z_2).
\end{equation}
Finally choosing $t\in\mathbb R^1$ and $\phi\in \mathbb R^1/2\pi \mathbb Z$ such that
$$\phi+\om t=0,\, -\phi+\om t=\pi,$$
we deduce from \eqref{four} that
\begin{equation}\label{7}
\mathcal N^1_{\mu}(z_1, -z_2, \bar z_1, -\bar z_2)=\mathcal N^1_{\mu}(z_1, z_2, \bar z_1, \bar z_2).
\end{equation}
Equations \eqref{5}, \eqref{6} and  \eqref{7} imply that the polynomial $\mathcal N^1_{\mu}(z_1, z_2, \bar z_1, \bar z_2)=z_1P_1(|z_1|^2, |z_2|^2)$ for some polynomial $P_1$.
Repeating the above reasoning, we also obtain that $\mathcal N^2_{\mu}(z_1, z_2, \bar z_1, \bar z_2)=z_2 P_2(|z_1|^2, |z_2|^2)$ for some polynomial $P_2$. Taking into account of \eqref{4}, we obtain that $P_1(|z_2|^2, |z_1|^2)=P_2(|z_1|^2, |z_2|^2)$.

To sum up, we obtain from the above analysis that the dynamics of \eqref{Brusselator} on the center manifolds has the following form
	\begin{equation}\label{normal form equation}
	\begin{cases}
	\frac{d}{dt}z_1=i\omega z_1+z_1 P_{\mu}(|z_1|^2, |z_2|^2)+\zeta(z_1, z_2, \bar z_1, \bar z_2, \mu),\\
	\frac{d}{dt}z_2=i\omega z_2+ z_2 P_{\mu}(|z_2|^2, |z_1|^2)+\zeta(z_2, z_1, \bar z_2, \bar z_1, \mu).
	\end{cases}
	\end{equation}
Due to tangency in Theorem \ref{normal form theorem}, the constant $c_0$ in the polynomial $P_{\mu}(|z_1|^2, |z_2|^2)=c_0+a\mu+b|z_1|^2+c|z_2|^2+h.o.t$ must be zero. Therefore, we obtain the following truncated form of \eqref{bifurcation} with the complex constants $a$, $b$ and $c$ to be determined:

         \begin{equation}
	\begin{cases}
	\frac{d}{dt}z_1=i\omega z_1+z_1(a\mu+b|z_1|^2+c|z_2|^2),\\
	\frac{d}{dt}z_2=i\omega z_2+ z_2(a\mu+b|z_2|^2+c|z_1|^2).
	\end{cases}
	\end{equation}

\subsection{Comparison of coefficients}\label{comparison of coefficients}

To show our main result, we will verify that the non-degeneracy conditions for the normal form dynamics can survive for certain parameter values. Therefore, we shall compute the constants $a, b$ and $c$ in the normal form dynamics. 

By center manifold theory, we decompose $v=z_1\xi_1+z_2\xi_2+\bar z_1\bar \xi_1+\bar z_2\xi_2+\Psi(z_1, z_2, \bar z_1,\bar z_2)$ where $\Psi(z_1, z_2, \bar z_1,\bar z_2,\mu)\in \mathbf Z_h$ is the center manifold reduction function. If we take into account of the normal form transformation step $w=z_1\xi_1+z_2\xi_2+\bar z_1\bar \xi_1+\bar z_2\xi_2+\Pi_{\mu}(z_1, z_2, \bar z_1, \bar z_2)$ where $\Pi_{\mu}(z_1, z_2, \bar z_1, \bar z_2)\in\mathbf Z_c$, we obtain another decomposition
$v=z_1\xi_1+z_2\xi_2+\bar z_1\bar \xi_1+\bar z_2\xi_2+\Pi_{\mu}(z_1, z_2, \bar z_1, \bar z_2)+\Psi(z_1, z_2, \bar z_1,\bar z_2)$. With a slight abuse of notation, we still write the sum $\Pi_{\mu}(z_1, z_2, \bar z_1, \bar z_2)+\Psi(z_1, z_2, \bar z_1,\bar z_2)$ as $\Psi(z_1, z_2, \bar z_1, \bar z_2, \mu)$ but now $\Psi(z_1, z_2, \bar z_1, \bar z_2, \mu)\in\mathbf Z$.

In order to determine the coefficients $a$, $b$ and $c$ in the truncated normal form, we expand the function $\Psi$:
\begin{equation}\label{Psi}
\Psi(z_1, z_2, \bar z_1, \bar z_2)=\sum_{p+q+r+s+l\geq 1}\Psi_{pqrsl}z_1^p \bar{z}_1^q z_2^r \bar{z}_2^s\mu^l.
\end{equation}
In this expansion, $\Psi_{pqrsl}$ are functions in the space $\mathbf Z$. Due to tangency, we know
$$\Psi_{10000}=\Psi_{01000}=\Psi_{00100}=\Psi_{00010}=0.$$ 
By conjugacy, we also have $\Psi_{pqrsl}=\bar{\Phi}_{qpsrl}$.

By flow invariance, we use the substitutions given by the normal form \eqref{normal form equation} and the expansion of $\Psi$ \eqref{Psi} in the bifurcation equation \eqref{bifurcation}. By comparing coefficients at orders $O(\mu)$, $O(\mu z_1)$ (or $O(\mu z_2)$), $O(|z_1|^2)$, $O(z_1^2)$, $O(|z_2|^2)$, $O(z_1z_2)$, $O(z_1 \bar z_2)$,
$O(z_1^2\bar z_1)$, $O(z_1 |z_2|^2)$, we obtain that 
\begin{equation}\label{operator equations}
\begin{cases}
(\mathbf{L}_{\be_1}+\mathbf R_{01})\Psi_{00001}=0,\\
a\xi+(i\omega-\mathbf L_{\beta_1})\Psi_{10100}=\mathbf R_{01}(\xi_1)+2\mathbf R_{20}(\xi_1, \Psi_{00001}),\\
(2i\om-\mathbf L_{\be_1})\Psi_{20000}=\mathbf R_{20}(\xi_1,\xi_1),\\
\mathbf L_{\be_1}\Psi_{11000}=-2\mathbf R_{20}(\xi_1,\bar\xi_1),\\
\Psi_{00110}=S\Psi_{11000},\\
(2i\om-\mathbf L_{\be_1})\Psi_{10100}=2\mathbf R_{20}(\xi_1,\xi_2),\\
\mathbf L_{\be_1}\Psi_{10010}=-2\mathbf R_{20}(\xi_1,\bar\xi_2),\\
b\xi_1+(i\omega-\mathbf L_{\be_1})\Psi_{21000}=2R_{20}(\xi_1, \Psi_{11000})+2R_{20}(\bar\xi_1, \Psi_{20000})+3R_{30}(\xi_1,\xi_1, \bar\xi_1),\\
c\xi_1+(i\omega-\mathbf L_{\be_1})\Psi_{10110}=2R_{20}(\xi_1, \Psi_{00110})+2R_{20}(\xi_2, \Psi_{10010})++2R_{20}(\bar\xi_2, \Psi_{10100})+6R_{30}(\xi_1,\xi_2, \bar\xi_2).
\end{cases}
\end{equation}

Among the above operator equations, the three relations containing $a$, $b$ and $c$ amount to saying that the following three
functions
\begin{equation}\label{vector 1}
-a\xi_1+\mathbf R_{01}(\xi_1)+2\mathbf R_{20}(\xi_1, \Psi_{00001}),
\end{equation}
\begin{equation}\label{vector 2}
-b\xi_1+2R_{20}(\xi_1, \Psi_{11000})+2R_{20}(\bar\xi_1, \Psi_{20000})+3R_{30}(\xi_1,\xi_1, \bar\xi_1),
\end{equation}
\begin{equation}\label{vector 3}
-c\xi_1+2R_{20}(\xi_1, \Psi_{00110})+2R_{20}(\xi_2, \Psi_{10010})++2R_{20}(\bar\xi_2, \Psi_{10100})+6R_{30}(\xi_1,\xi_2, \bar\xi_2)
\end{equation}
are all elements in the range of the operator $i\omega-\mathbf L_{\be_1}$.
By the kernel and range relation for a linear operator and its conjugate, we know
$\overline{R(i\omega-\mathbf L_{\be_1})}=\big(\ker\, (i\omega-\mathbf L_{\be_1})^*\big)^{\perp}$ where
the symbols $R(i\omega-\mathbf L_{\be_1})$ and $\ker \,(i\omega-\mathbf L_{\be_1})^*$ represent the range of $i\omega-\mathbf L_{\be_1}$ and
kernel of the linear operator $(i\omega-\mathbf L_{\be_1})^*$ respectively.
Recalling the definition of $\xi_1^*$ by \eqref{xi star}, we know that the three vectors in \eqref{vector 1}, \eqref{vector 2}, \eqref{vector 3} are all perpendicular to $\xi_1^*$ in $L^2(\mathbb T)\times L^2(\mathbb T)$. Therefore, by orthogonality, we have
\begin{equation}\label{a}
a=\langle \mathbf R_{01}(\xi_1)+2\mathbf R_{20}(\xi_1, \Psi_{00001}), \xi_1^{*}  \rangle;
\end{equation}

\begin{equation}\label{b}
b=\langle 2\mathbf R_{20}(\xi_1, \Psi_{11000})+2\mathbf R_{20}(\bar\xi_1, \Psi_{20000})+3\mathbf R_{30}(\xi_1,\xi_1, \bar\xi_1), \xi_1^{*}  \rangle;
\end{equation}

\begin{equation}\label{c}
c=\langle 2\mathbf R_{20}(\xi_1, \Psi_{00110})+2\mathbf R_{20}(\xi_2, \Psi_{10010})+2\mathbf R_{20}(\bar\xi_2, \Psi_{10100})+6\mathbf R_{30}(\xi_1,\xi_2, \bar\xi_2), \xi_1^{*}  \rangle.
\end{equation}

Now we proceed to examine the coefficients $a, b$ and $c$, which is crucial to test the non-degeneracy of the 
spatial-temporal oscillations for our model under investigation. 

First, we aim to obtain the parameter $a$. We can compute $a$ either by asymptotic analysis or by the above functional equation. Here we demonstrate both devices as one device can be regarded as a verification of the computational result of the other. 

Let us first calculate $a$ by asymptotic analysis. The parameter $a$ is the coefficient of $\mu$  in the spectral branch of the operator $\mathbf{L}_{\be}$ corresponding to the eigenvalue $i\om$ in the limit $\mu\rightarrow 0$. The eigenvalues of the operator $\mathbf L_{\be}=\begin{pmatrix} \delta_1\partial^2_x +\be_1+\mu-1 & \al^2\\ -\be_1-\mu & \delta_2\partial^2_x -\al^2\end{pmatrix}$ corresponding to $\pm i\om$ are given by the roots of the polynomial
$$
P_1(\lambda,\be)=\lambda^2+(\be_1-\be_1-\mu)\lambda+\om^2-\de_2\mu=\lambda^2-\mu\lambda+\om^2-\de_2\mu.
$$ 
Denote the two branches of eigenvalues by $\lambda_{\pm}(\mu)$, we obtain by the quadratic formula that
$$\lambda_{\pm}(\mu)=\frac{1}{2}\mu\pm i\sqrt{\om^2-\de_2\mu-\mu^2/4}.$$
Therefore, the asymptotic expansion of the branch corresponding to $i\om$ is given by
$$\lambda_{+}(\mu)=i\om+(\frac{1}{2}-i\frac{\de_2}{2\om})\mu+O(\mu^2).$$
Consequently, we find that
$$a=\frac{1}{2}-i\frac{\de_2}{2\om}.$$

Now we demonstrate the calculation of $a$ through \eqref{a}. From the relation $(\mathbf{L}_{\be_1}+\mathbf R_{01})\Psi_{00001}=0$, we can compute $\Psi_{00001}$. In fact, to compute $\Psi_{00001}$, we have to seek solutions of the form $e^{0ix}V$ with $V$ being a complex vector. Put this ansatz into the operator equation $(\mathbf{L}_{\be_1}+\mathbf R_{01})\Psi_{00001}=0$, we obtain the following algebraic equation on the complex vector $V$
\begin{equation*}
\begin{pmatrix}\be_1-2 & \al^2\\ -\be_1+1 & -\al^2 \end{pmatrix}
V=\begin{pmatrix}0\\0 \end{pmatrix},
\end{equation*}
which only has the trivial solution. Therefore, we find that $\Psi_{00001}=0$, and consequently obtain that
\begin{equation}\label{a formula}
a=\langle \mathbf R_{01}(\xi_1)+2\mathbf R_{20}(\xi_1, \Psi_{00001}), \xi_1^{*}  \rangle=\langle \mathbf R_{01}(\xi_1), \xi_1^{*}  \rangle=\frac{1}{2}-i\frac{\de_2}{2\om}.
\end{equation}

Second, we calculate the parameter $b$. For this purpose, we first compute $\Psi_{11000}$ and $\Psi_{20000}$.

As we know that $\Psi_{11000}$ is determined by the operator equation
$\mathbf L_{\be_1}\Psi_{11000}=-2\mathbf R_{20}(\xi_1,\bar\xi_1)$,
we first examine the right hand side term of the operator equation. Simple computation yields
$$-2\mathbf R_{20}(\xi_1,\bar{\xi_1})=\frac{4(\al^2+\de_2)-2\be_1}{\al}\begin{pmatrix}1\\-1\end{pmatrix}.$$
Due to the form of $-2\mathbf R_{20}(\xi_1,\bar{\xi_1})$, we need seek a solution of the form $e^{i0x}V$ with $V$ being a complex vector. Using the substitution  $\Psi_{11000}=e^{i0x}V$ in the equation $\mathbf L_{\be_1}\Psi_{11000}=-2\mathbf R_{20}(\xi_1,\bar\xi_1)$, we arrive at the following algebraic equation on $V$:
$$
\begin{pmatrix}\be_1-1 & \al^2\\ -\be_1 & -\al^2 \end{pmatrix}V=\frac{4(\al^2+\de_2)-2\be_1}{\al}\begin{pmatrix}1\\-1\end{pmatrix}.
$$
As $0\in \rho(\mathbf L_{\be_1})$, the resolvent set of $\mathbf L_{\be_1}$, we know
$\det \begin{pmatrix}\be_1-1 & \al^2\\ -\be_1 & -\al^2 \end{pmatrix}=P_0(0)\not=0$.
Therefore, the vector $V$ can be solved uniquely:
$$
V=\begin{pmatrix}\be_1-1 & \al^2\\ -\be_1 & -\al^2 \end{pmatrix}^{-1}\frac{4(\al^2+\de_2)-2\be_1}{\al}\begin{pmatrix}1\\-1\end{pmatrix}=\frac{4(\al^2+\de_2)-2\be_1}{\al^3}\begin{pmatrix}0\\1\end{pmatrix}.
$$
Consequently, we find that
\begin{equation}\label{11000}
\Psi_{11000}=\frac{4(\al^2+\de_2)-2\be_1}{\al^3}\begin{pmatrix}0\\1\end{pmatrix}.
\end{equation}

The computation of $\Psi_{20000}$ can be proceeded similarly. By examining the right hand side term of the operator equation $$(2i\om-\mathbf L_{\be_1})\Psi_{20000}=\mathbf R_{20}(\xi_1,\xi_1),$$
we find that
$$\mathbf R_{20}(\xi_1,\xi_1)=\frac{-2(\al^2+\de_2)+\be_1+2i\om}{\al}e^{2ix}\begin{pmatrix}1\\-1\end{pmatrix}.$$
Inspecting the form of $\mathbf R_{20}(\xi_1,\xi_1)$ above, we shall seek a solution of the form $e^{2ix}V$ with $V$ being a complex vector. Inserting $\Psi_{20000}=e^{2ix}V$ to the above operator equation, we deduce the following algebraic equation on the complex vector  $V$:
$$\Big\{2i\om I_2-\begin{pmatrix}-4\de_1+\be_1-1 & \al^2\\ -\be_1 & -4\de_2-\al^2 \end{pmatrix}\Big\}V=\frac{-2(\al^2+\de_2)+\be_1+2i\om}{\al}\begin{pmatrix}1\\-1\end{pmatrix}.$$
As $2i\om\in \rho(\mathbf L_{\be_1})$, we know
$$\det\Big\{2i\om I_2-\begin{pmatrix}-4\de_1+\be_1-1 & \al^2\\ -\be_1 & -4\de_2-\al^2 \end{pmatrix}\Big\}=P_2(2i\om)\not=0.$$
Therefore, the vector $V$ can be solved uniquely as
\begin{align*}
V&=\frac{1}{P_2(2i\om)}\frac{-2(\al^2+\de_2)+\be_1+2i\om}{\al}
\begin{pmatrix}2i\om+4\de_2+\al^2 & \al^2 \\ -\be_1 & 2i\om+4\de_1-\be_1+1 \end{pmatrix}\begin{pmatrix}1\\-1\end{pmatrix}\\
&=\frac{1}{P_2(2i\om)}\frac{-2(\al^2+\de_2)+\be_1+2i\om}{\al}
\begin{pmatrix}2i\om+4\de_2 \\ -2i\om-4\de_1-1 \end{pmatrix}.
\end{align*}
Therefore, we obtain
\begin{equation}\label{20000}
\Psi_{20000}=\frac{1}{P_2(2i\om)}\frac{-2(\al^2+\de_2)+\be_1+2i\om}{\al}
e^{2ix}\begin{pmatrix}2i\om+4\de_2 \\ -2i\om-4\de_1-1 \end{pmatrix}.
\end{equation}

Using \eqref{11000} and \eqref{20000}, we can now compute all the terms 
$3\mathbf R_{30}(\xi_1, \xi_1, \bar{\xi_1})$, $2\mathbf R_{20}(\xi_1, \Psi_{11000})$
and $2\mathbf R_{20}(\bar\xi_1, \Psi_{20000})$ in \eqref{b}. Direct computations yield:
$$3\mathbf R_{30}(\xi_1, \xi_1, \bar{\xi_1})=\frac{-3(\al^2+\de_2)+i\om}{\al^2}e^{ix}\begin{pmatrix}1\\-1 \end{pmatrix},$$
$$2\mathbf R_{20}(\xi_1, \Psi_{11000})=\frac{8(\al^2+\de_2)-4\be_1}{\al^2}e^{ix}\begin{pmatrix}
1\\-1\end{pmatrix},$$
\begin{align*}
2\mathbf R_{20}(\bar\xi_1, \Psi_{20000})=&\frac{2}{P_2(2i\om)}\frac{-2(\al^2+\de_2)+\be_1+2i\om}{\al}e^{ix}\begin{pmatrix}1\\-1\end{pmatrix}
\times\\&\Big(-\al(2i\om+4\de_1+1)-\frac{2i\om+4\de_2}{\al}(i\om-1-\de_1)\Big).
\end{align*}
Gleaning the above information and noticing that
\begin{equation}\label{product}
\Big\langle e^{ix}\begin{pmatrix}1\\-1\end{pmatrix}, \frac{i\al^2}{4\pi\om}e^{ix} \begin{pmatrix}\frac{\de_2+\al^2-i\om}{\al^2}\\1\end{pmatrix}\Big\rangle=-\frac{i(\de_2+i\om)}{2\om},
\end{equation}
we obtain that
\begin{align}\label{b expression}
b=&\langle 2\mathbf R_{20}(\xi_1, \Psi_{11000})+2\mathbf R_{20}(\bar\xi_1, \Psi_{20000})+3\mathbf R_{30}(\xi_1,\xi_1, \bar\xi_1), \xi_1^{*}  \rangle \nonumber\\
=&-\frac{i(\de_2+i\om)}{2\om}\Big\{\frac{5(\al^2+\de_2)-4\be_1+i\om}{\al^2}+\frac{2}{P_2(2i\om)}\frac{-2(\al^2+\de_2)+\be_1+2i\om}{\al}\nonumber\\
& \times\Big(-\al(2i\om+4\de_1+1)-\frac{2i\om+4\de_2}{\al}(i\om-1-\de_1)\Big) \Big\}\nonumber\\
=&\frac{\om-i\de_2}{2\om\al^2}\Big\{\Big(5(\al^2+\de_2)-4\be_1+i\om\Big)+\frac{2}{P_2(2i\om)}\Big(-2(\al^2+\de_2)+\be_1+2i\om\Big)\nonumber\\
& \times\Big(-\al^2(2i\om+4\de_1+1)-(2i\om+4\de_2)(i\om-1-\de_1)\Big) \Big\}.
\end{align}

In oder to simplify the above expression, we shall make some simple but tedious computations. We first observe that
\begin{align}
&\Big(-2(\al^2+\de_2)+\be_1+2i\om\Big)
\times\Big(-\al^2(2i\om+4\de_1+1)-(2i\om+4\de_2)(i\om-1-\de_1)\Big) \nonumber\\
&=(\be_1-2\al^2-2\de_2)(2\om^2+4\de_2+4\de_1\de_2-\al^2-4\al^2\de_1)-2\om^2(2+2\de_1-4\de_2-2\al^2)\nonumber\\
&+i\om\Big( (\be_1-2\al^2-2\de_2)(2+2\de_1-4\de_2-2\al^2)+2(2\om^2+4\de_2+4\de_1\de_2-\al^2-4\al^2\de_1)  \Big)\nonumber\\
&:=N_r+i\om N_i,
\end{align}
where we have used the following convention for the real quantities $N_r$ and $N_i$:
\begin{equation}\label{Nr}
N_r=(\be_1-2\al^2-2\de_2)(2\om^2+4\de_2+4\de_1\de_2-\al^2-4\al^2\de_1)-2\om^2(2+2\de_1-4\de_2-2\al^2)
\end{equation}
and
\begin{equation}\label{Ni}
N_i=(\be_1-2\al^2-2\de_2)(2+2\de_1-4\de_2-2\al^2)+2(2\om^2+4\de_2+4\de_1\de_2-\al^2-4\al^2\de_1).
\end{equation}

Denote by $B_r$ and $\om B_i$ the real part and imaginary part of 
\begin{align*}
\frac{2}{P_2(2i\om)}\Big(-2(\al^2+\de_2)+\be_1+2i\om\Big)
\Big(-\al^2(2i\om+4\de_1+1)-(2i\om+4\de_2)(i\om-1-\de_1)\Big)
\end{align*}
respectively. To make further computations, we first write $P_2(2i\om)$ explicitly. Noticing the expressions of $P_n(\lambda)$ and $\om$, we obtain
$$P_2(2i\om)=-3\al^2+12\de_1\de_2+6i(\de_1+\de_2)\om,$$
which is a complex number. Therefore, we have
\begin{equation}\label{P2omega}
\frac{1}{P_2(2i\om)}=-3\frac{\al^2-4\de_1\de_2+2(\de_1+\de_2)\om i}{(\al^2-4\de_1\de_2)^2+4(\de_1+\de_2)^2\om^2}.
\end{equation}
Then, by direct computations, the real quantities $B_r$ and $B_i$ are given by
\begin{equation}\label{Br}
B_r=\frac{-6}{(\al^2-4\de_1\de_2)^2+4(\de_1+\de_2)^2\om^2}
\Big((\al^2-4\de_1\de_2)N_r-2\om^2(\de_1+\de_2) N_i\Big),
\end{equation}
\begin{equation}\label{Bi}
B_i=\frac{-6}{(\al^2-4\de_1\de_2)^2+4(\de_1+\de_2)^2\om^2}\Big((\al^2-4\de_1\de_2)N_i+2(\de_1+\de_2) N_r
\Big)
\end{equation}
in which $N_r$ and $N_i$ are given by \eqref{Nr} and \eqref{Ni} respectively.

Collecting the information above, we obtain the following expression of $b$ from \eqref{b expression}:
\begin{align}\label{b formula}
b=&\frac{\om-i\de_2}{2\om\al^2}\Big\{\Big(5(\al^2+\de_2)-4\be_1+i\om\Big)+(B_r+i \om B_i) \Big\}\nonumber\\
&=\frac{\om-i\de_2}{2\om\al^2}\Big(5(\al^2+\de_2)-4\be_1+B_r+i\om(1+B_i)\Big)\nonumber\\
&=\frac{\om}{2\om\al^2}(5\al^2+5\de_2-4\be_1+B_r+\de_2+\de_2 B_i)\nonumber\\
&+i\frac{1}{2\om\al^2}(\om^2+\om^2 B_i-5\de_2\al^2-5\de_2^2+4\de_2\be_1-\de_2 B_r).
\end{align}

Third, we compute the parameter $c$. For this purpose, we first determine $\Psi_{10010}$, $\Psi_{00110}$ and
$\Psi_{10100}$.

To determine $\Psi_{10010}$ from the operator equation
$\mathbf L_{\be_1}\Psi_{10010}=-2\mathbf R_{20}(\xi_1,\bar\xi_2)$,
we first note that
$$-2\mathbf R_{20}(\xi_1,\bar\xi_2)=\frac{4(\al^2+\de_2)-2\be_1}{\al}e^{2ix}\begin{pmatrix}1\\-1 \end{pmatrix}.$$
Due to the form of $-2\mathbf R_{20}(\xi_1,\bar\xi_2)$ above, we shall seek a solution of the form $\Psi_{10010}=e^{2ix}V$. Putting this ansatz into the equation $\mathbf L_{\be_1}\Psi_{10010}=-2\mathbf R_{20}(\xi_1,\bar\xi_2)$, we obtain an algebraic equation on $V$:
$$
\begin{pmatrix}-4\de_1+\be_1-1 & \al^2 \\-\be_1 & -4\de_2-\al^2\end{pmatrix}V=\frac{4(\al^2+\de_2)-2\be_1}{\al}\begin{pmatrix}1\\-1\end{pmatrix}.
$$
As $\begin{pmatrix}-4\de_1+\be_1-1 & \al^2 \\-\be_1 & -4\de_2-\al^2\end{pmatrix}=P_2(0)\not=0$, the complex vector $V$ can be solved uniquely as 
$$V=\frac{1}{P_2(0)}\frac{4(\al^2+\de_2)-2\be_1}{\al}\begin{pmatrix}-4\de_2 \\ 4\de_1+1 \end{pmatrix}.$$
Therefore, we have
\begin{equation}\label{10010}
\Psi_{10010}=\frac{1}{P_2(0)}\frac{4(\al^2+\de_2)-2\be_1}{\al}e^{2ix}\begin{pmatrix}-4\de_2 \\ 4\de_1+1 \end{pmatrix}.
\end{equation}

Now we look for $\Psi_{10100}$ which is determined by the operator equation 
$$(2i\om-\mathbf L_{\be_1})\Psi_{10100}=2\mathbf R_{20}(\xi_1,\xi_2).$$
First, we find by direct computation that 
$$2\mathbf R_{20}(\xi_1,\xi_2)=\frac{4(-\al^2-\de_2+i\om)+2\be_1}{\al}\begin{pmatrix}1\\-1\end{pmatrix},$$
which implies that $\Psi_{10100}$ has the form $\Psi_{10100}=e^{0ix}V$ for some complex vector $V$.
As $2i\om\in\rho(\mathbf L_{\be_1})$, we can solve $\Psi_{10100}$ uniquely as before and obtain 
\begin{equation}\label{10100}
\Psi_{10100}=\frac{1}{P_0(2i\om)}\frac{4(-\al^2-\de_2+i\om)+2\be_1}{\al}
\begin{pmatrix}2i\om\\ -1-2i\om \end{pmatrix}.
\end{equation}

With \eqref{11000}, \eqref{10010} and \eqref{10100} at hand, we are now ready to compute the terms $6\mathbf R_{30}(\xi_1, \xi_2, \bar{\xi_2})$, $\Psi_{00110}$, $2\mathbf R_{20}(\xi_1, \Psi_{00110})$, $2\mathbf R_{20}(\bar\xi_2, \Psi_{10100})$ and $2\mathbf R_{20}(\bar\xi_2, \Psi_{10100})$. Direct computations yield:
$$6\mathbf R_{30}(\xi_1, \xi_2, \bar{\xi_2})=\frac{2}{\al^2}\Big(-3(\al^2+\de_2)+i\om\Big)e^{ix}\begin{pmatrix}1\\-1\end{pmatrix},$$
$$\Psi_{00110}=S\Psi_{11000}=\Psi_{11000}=\frac{4(\al^2+\de_2)-2\be_1}{\al^3}\begin{pmatrix}0\\1\end{pmatrix},$$
$$2\mathbf R_{20}(\xi_1, \Psi_{00110})=\frac{8(\al^2+\de_2)-4\be_1}{\al^2}e^{ix}\begin{pmatrix}1\\-1\end{pmatrix},$$
\begin{align*}
2\mathbf R_{20}(\xi_2, \Psi_{10010})=&\frac{4}{P_2(0)}\frac{2(\al^2+\de_2)-\be_1}{\al}
e^{ix}\begin{pmatrix}1\\-1\end{pmatrix}\times\Big(\al(4\de_1+1)-\frac{4\de_2}{\al}(i\om+1+\de_1)\Big),
\end{align*}
\begin{align*}
2\mathbf R_{20}(\bar\xi_2, \Psi_{10100})=&\frac{4}{P_0(2i\om)}\frac{2(-\al^2-\de_2+i\om)+\be_1}{\al}
\times\Big(-\al(2i\om+1)+\frac{2i\om}{\al}(-i\om+1+\de_1)\Big) \\& e^{ix}\begin{pmatrix}1\\-1\end{pmatrix}.
\end{align*}

Gleaning the information above and again noticing \eqref{product}, we have
\begin{align}\label{c expression}
c=&\langle 2\mathbf R_{20}(\xi_1, \Psi_{00110})+2\mathbf R_{20}(\xi_2, \Psi_{10010})+2\mathbf R_{20}(\bar\xi_2, \Psi_{10100})+6\mathbf R_{30}(\xi_1,\xi_2, \bar\xi_2), \xi_1^{*}  \rangle\nonumber\\
=&-\frac{i(\de_2+i\om)}{2\om}\Big\{\frac{2(\al^2+\de_2)-4\be_1+2i\om}{\al^2}+C_1+C_2 \Big\},
\end{align}
where $C_1$ and $C_2$ are given by
\begin{align}\label{C1}
C_1&=\frac{4}{P_2(0)}\frac{2(\al^2+\de_2)-\be_1}{\al}
\times\Big(\al(4\de_1+1)-\frac{4\de_2}{\al}(i\om+1+\de_1)\Big)\nonumber\\
&=\frac{4}{P_2(0)}\frac{2(\al^2+\de_2)-\be_1}{\al}
\times\Big(\al(4\de_1+1)-\frac{4\de_2(1+\de_1)}{\al}-\frac{4\de_2\om}{\al}i\Big),
\end{align}
\begin{align}\label{C2}
C_2&=\frac{4}{P_0(2i\om)}\frac{2(-\al^2-\de_2+i\om)+\be_1}{\al}
\times\Big(-\al(2i\om+1)+\frac{2i\om}{\al}(-i\om+1+\de_1)\Big)\nonumber\\
&=\frac{4}{P_0(2i\om)}\frac{\be_1-2(\al^2+\de_2)+2i\om}{\al}
\times\Big((-\al+\frac{1+\de_1}{\al})2i\om+(\frac{2\om^2}{\al}-\al)\Big)\nonumber\\
&=\frac{4}{P_0(2i\om)}\frac{\be_1-2(\al^2+\de_2)+2i\om}{\al^2}
\times\Big((-\al^2+1+\de_1)2i\om+(2\om^2-\al^2)\Big)\nonumber\\
&=\frac{4}{P_0(2i\om)}\frac{1+\de_1-\de_2-\al^2+2i\om}{\al^2}
\times\Big((-\al^2+1+\de_1)2i\om+(2\om^2-\al^2)\Big)\nonumber\\
&=\frac{4}{P_0(2i\om)\al^2} \Big\{ (1+\de_1-\de_2-\al^2)(2\om^2-\al^2)-4(1+\de_1-\al^2)\om^2 +\nonumber\\
&\Big( (2\om^2-\al^2)+(1+\de_1-\de_2-\al^2)(1+\de_1-\al^2)   \Big)2i\om  \Big\}.
\end{align}
Easy calculation shows that
\begin{equation}\label{P20}
P_2(0)=\al^2(4\de_1-4\de_2+1)+12\de_1\de_2-4\de_2^2,
\end{equation}
which is a real number.

In view of \eqref{P2omega}, we can write $C_2$ as follows
$$C_2=\frac{-12}{\al^2[(\al^2-4\de_1\de_2)^2+4(\de_1+\de_2)^2\om^2]}(C_{2r}+iC_{2i})$$
with $C_{2r}$ and $C_{2i}$ given by
\begin{align}\label{C2r}
C_{2r}&=(\al^2-4\de_1\de_2)\Big((1+\de_1-\de_2-\al^2)(2\om^2-\al^2)-4(1+\de_1-\al^2)\om^2\Big)\nonumber\\
&-4\om^2(\de_1+\de_2)\Big( 2\om^2-\al^2+(1+\de_1-\de_2-\al^2)(1+\de_1-\al^2) \Big),
\end{align}
\begin{align}\label{C2i}
C_{2i}&=2\om\Big\{(\de_1+\de_2)\Big((1+\de_1-\de_2-\al^2)(2\om^2-\al^2)-4(1+\de_1-\al^2)\om^2\Big)\nonumber\\
&+(\al^2-4\de_1\de_2)\Big( 2\om^2-\al^2+(1+\de_1-\de_2-\al^2)(1+\de_1-\al^2) \Big)\Big\}.
\end{align}

Defining $Q_r$ and $Q_i$ as
\begin{align}\label{Qr}
Q_r&=2(\al^2+\de_2)-4\be_1+\frac{4}{P_2(0)}(2\al^2+2\de_2-\be_1)(4\al^2\de_1+\al^2-4\de_2-4\de_1\de_2)\nonumber\\
&-\frac{12 C_{2r}}{(\al^2-4\de_1\de_2)^2+4(\de_1+\de_2)^2\om^2},
\end{align}
\begin{equation}\label{Qi}
Q_i=2\om-\frac{16\de_2\om}{P_2(0)}(2\al^2+2\de_2-\be_1)-\frac{12 C_{2i}}{(\al^2-4\de_1\de_2)^2+4(\de_1+\de_2)^2\om^2},
\end{equation}
we obtain the following expression of $c$ from \eqref{c expression}: 
\begin{equation}\label{c formula}
c=\frac{\om-i\de_2}{2\om \al^2}(Q_r+iQ_i)=\frac{1}{2\om\al^2}\Big( (\om Q_r+\de_2 Q_i) +(\om Q_i-\de_2 Q_r)i \Big).\end{equation}

It is the real parts of the parameters $a$, $b$ and $c$ that are crucial for our analysis of the dynamics driven by \eqref{bifurcation} or equivalently \eqref{reaction}. Therefore, we summarize the information on \eqref{a formula}, \eqref{b formula} and \eqref{c formula} in the following proposition.

\begin{proposition} \label{proposition abc}
The real parts for the parameters $a$, $b$ and $c$ in $P_{\mu}$ in the normal form equation \eqref{normal form equation} for \eqref{bifurcation}
are given by
\begin{equation*}
\mbox{Re}\, a=\frac{1}{2}, \mbox{Re}\, b=\frac{1}{2\al^2}(5\al^2+5\de_2-4\be_1+B_r+\de_2+\de_2 B_i),
\mbox{Re}\, c=\frac{1}{2\om\al^2}(\om Q_r+\de_2 Q_i) 
\end{equation*}
where $B_r$, $B_i$, $Q_r$ and $Q_i$ are given by \eqref{Br}, \eqref{Bi}, \eqref{Qr} and \eqref{Qi} respectively.
\end{proposition}


\section{Analysis of Dynamics}

The functional equalities in \eqref{operator equations}
must hold once a choice of function spaces $\mathbf Z, \mathbf Y$ and $\mathbf X$ is made. Here we shall remark that $\mathbf X$ can not be chosen as $L^2_0(\mathbf T)\times L^2_0(\mathbf T)$ (hence there can not be mean-zero restriction in the space $\mathbf X$) or else the functional relations
$\mathbf L_{\be_1}\Psi_{11000}=-2\mathbf R_{20}(\xi_1,\bar\xi_1),\,\Psi_{00110}=S\Psi_{11000}$
and
$(2i\om-\mathbf L_{\be_1})\Psi_{10100}=2\mathbf R_{20}(\xi_1,\xi_2)$
have no solutions. This is determined by the structure of \eqref{1a}-\eqref{1d} or equivalently the form of \eqref{bifurcation}.
The mean-zero requirement in the space $H^2_0(\mathbf T)\times H^2_0(\mathbf T)$ would exclude the constant term in the Fourier expansion of a function. See also the computations of $\Psi_{11000}$ and $\Psi_{10100}$ in last section. As mean-zero perturbations of the uniform state correspond to different distributions of the same amount of chemical reactants, these perturbations represent the inhomogeneity of the concentrations of
the chemical reactants. A direct consequence of our analysis is that the wave structure we demonstrated for the system \eqref{Brusselator} is indeed not solely due to the inhomogeneity of the concentrations. 

The analysis of the dynamics \eqref{normal form equation} is now routine (see e.g., \cite{LY, Yao}). We produce the analysis here for completeness and for the convenience of readers. For our purpose, it is sufficient to consider the truncated normal form at order three as the wave structure is persistent for \eqref{rd} and \eqref{normal form equation} by implicit function theorem.

To proceed, we introduce the polar coordinates $z_1=r_1e^{i\theta_1}$ and $z_2=r_2e^{i\theta_2}$,
the truncated normal form at cubic order becomes
\begin{equation}\label{rd}
\begin{cases}
\frac{dr_1}{dt}=r_1 \big((\mbox{Re}\,a)\mu+(\mbox{Re}\,b) r_1^2 +(\mbox{Re}\,c)  r_2^2\big),\\
\frac{dr_2}{dt}=r_2 \big((\mbox{Re}\,a)\mu+(\mbox{Re}\,b) r_2^2 +(\mbox{Re}\,c) r_1^2\big),\\
\frac{d\theta_1}{dt}=\omega+(\mbox{Im}\,a)\mu +(\mbox{Im}\,b) r_1^2+(\mbox{Im}\,c) r_2^2,\\
\frac{d\theta_2}{dt}=\omega+(\mbox{Im}\,a)\mu +(\mbox{Im}\,b) r_2^2+(\mbox{Im}\,c) r_1^2.\\
\end{cases}
\end{equation}

The equations on $(r_1, r_2)$ and $(\theta_1, \theta_2)$ decouple. The point to do the analysis is that we can first seek bifurcated nontrivial equilibria in the radial equations by assuming $r_1\equiv0$ or $r_2\equiv0$, or $r_1=r_2$ and these bifurcated equilibria correspond to time periodic solutions to the system \eqref{rd} due to the rotations given by $(\theta_1, \theta_2)$-equation. This point is guided by the symmetry of the normal form equation, in particular, the system is invariant if one interchanges $r_1$ and $r_2$, which is induced by $SO(2)$-symmetry.
For us, we always do analysis in a small neighborhood of $\mu=0$ in $\R^1$. 

We shall first analyze the radial equations. As we have obtained that $\mbox{Re}\,a=\frac{1}{2}$, we consider the auxiliary real function $f(r)=\frac{1}{2}\mu+\Lambda r^2$ defined on $[0,\epsilon]$ for some small $\epsilon>0$ where $\Lambda$ is a fixed real constant whose sign is important for us. If $\mu \Lambda>0$, $f$ has no roots in the interval $(0, +\infty)$. If $\mu \Lambda<0$, $f$ has a nontrivial root $r=\sqrt{-\frac{\mu}{2\Lambda}}=O(|\mu|)^{1/2})$. Now, consider the right hand sides of the radial equations of $r_1$ and $r_2$. If we let $r_2\equiv 0$, then the right hand side of equation on $r_1$ contains a factor with the same structure as $f(r)$. We can consider the $r_2$ equation similarly. Hence the above analysis for $f(r)$ applies.  Let $r_{*}(\mu)=\sqrt{-\frac{\mu}{2\mbox{Re}\,b}}$. We conclude that besides the trivial solution $(0,0)$, there are bifurcated solutions of the forms $(r_{*}(\mu), 0)$, and $(0, r_{*}(\mu))$. Further, if $\mbox{Re}\,b+\mbox{Re}\,c$ does not vanish, we may consider the situation $r_1\equiv r_2$. In such a situation, the $r_1$ and $r_2$ equations are the same and both contain a factor of the form $f(r)$ in the right hand sides. Therefore, there are bifurcated solutions of the form $( r_{*}(\mu), r_{*}(\mu))$ with $r_{*}(\mu)=\sqrt{-\frac{\mu}{2(\mbox{Re}\,b+\mbox{Re}\,c)}}$.

Now, we are in a position to analyze the two angular equations. From the above analysis, we know that all the three families of bifurcated solutions have magnitude $O(|\mu|^{1/2})$. As a consequence, we could arrange that $\frac{d\theta_1}{dt}\geq\frac{\omega}{2}>0$ and $\frac{d\theta_2}{dt}\geq\frac{\omega}{2}>0$ when $|\mu|$ remains small, which enables us to conclude that all the three families of bifurcated equilibria correspond to genuine time periodic waves of the system \eqref{rd}. The equilibria $(r_{*}(\mu), 0)$ and $(0, r_{*}(\mu))$ correspond to rotating waves on $r_1$-axis and $r_2$-axis, which is the same as for the Hopf bifurcation with $SO(2)$ symmetry. The symmetry $S$ plays the role of exchanging the two axes, i.e., exchanging the rotating waves corresponding to $r_2=0$ into the rotating waves corresponding to $r_1=0$. The equilibria $(r_{*}(\mu), r_{*}(\mu))$ with $r_1=r_2$ correspond to standing waves, another class of bifurcating periodic solutions. These waves correspond to a torus of solutions of the normal form
\begin{align*}
V_0(t,\mu, \phi_1, \phi_2)&=r_{*}(\mu)\big(e^{i\omega_{*}(\mu)t+\phi_1}\xi_1+ e^{i\omega_{*}(\mu)t+\phi_2}\xi_2\big)\\
&+r_{*}(\mu)\big( e^{-(i\omega_{*}(\mu)t+\phi_1)}\bar \xi_1+ e^{-(i\omega_{*}(\mu)t+\phi_2)}\bar\xi_2\big)
\end{align*}
for any $(\phi_1,\phi_2)\in\mathbb{R}^2$, which induces a torus
of solutions $U(t,\mu, \phi_1,\phi_2)$ in $\mathbf Y$ of the nonlinear
perturbation system \eqref{bifurcation}. The $\omega_{*}(\mu)$ is the phase
function determined by the $(\theta_1, \theta_2)$-equation in system
\eqref{rd} such that $\omega_{*}(0, 0)=\omega$. These standing waves in
addition possess the following symmetry
$$R(\phi_2-\phi_1)SU(t,\mu,\phi_1,\phi_2)=U(t,\mu,\phi_1,\phi_2), \,\, R(2\pi)U(t,\mu, \phi_1,\phi_2)=U(t,\mu,\phi_1,\phi_2),$$
$$R(\pi)U(t,\mu,\phi_1,\phi_2)=U(t+\frac{\pi}{\omega_{*}(\mu)}, \mu, \phi_1, \phi_2), \,\, SU(t,\mu,\phi_1, \phi_1)=U(t,\mu_1,\phi_1, \phi_1).$$
The analysis of the stability of the three families of bifurcated waves are straightforward by examining the signs of $\mbox{Re}\,b$ and $\mbox{Re}\, b\pm\mbox{Re}\, c$.

\section*{Appendix}
\appendix

In this appendix we collect the center manifold theorem and normal form theorem. Interested readers can refer to 
the works \cite{BB, Ca, GSS, He, HI, LY, MW, Yao} and the references therein for further information.
\section{Center manifold theorem}
Here we recall a version of the center manifold theorem with parameters and symmetry adapted to our study.

 \begin{theorem} (Center manifold theorem, see \cite{HI, LY, Yao})\label{center manifold theorem}
	Let the inclusions in the Banach space triplet $\mathbf Z \subset \mathbf Y\subset \mathbf X$ be continuous. Consider a differential equation in the Banach space $\mathbf X$ of the form
	$$\frac{dv}{dt}=\mathbf {L}v+\mathbf{R}(v,\mu)$$ and assume that
	
	(1) (form of nonlinearity) for some $k\geq 2$, there exist neighborhoods
	$\mathcal{V}_{v}\subset\mathbf Z$ and
	$\mathcal{V_\mu}\subset\mathbb{R}^m$ of $(0,0)$ such that
	$\mathbf{R}\in C^k(\mathcal{V}_u\times \mathcal{V_{\mu}},
	\mathbf Y)$ and
	$$\mathbf{R}(0,0)=0,\,\, D_u\mathbf{R}(0,0)=0.$$
	
	(2) (spectral decomposition) $\mathbf {L}:\mathbf Z\mapsto\mathbf X$ is a bounded linear map
	and there exists some constant $\gamma>0$ such that
	$$\inf\{\mbox{Re}\,\lambda; \lambda\in \sigma_u(\mathbf L)\}>\gamma,\,\, \sup\{\mbox{Re}\,\lambda; \lambda\in \sigma_s(\mathbf L)\}<-\gamma,$$
	and the set $\sigma_c(\mathbf L)$ consists of a finite number of eigenvalues with finite algebraic multiplicities.
	
	(3) (resolvent estimates) for Hilbert space triplet $\mathbf Z \subset \mathbf Y\subset \mathbf X$, assume
	there exists a positive constant $\omega_0>0$ such that $i\lambda\in \rho(\mathcal L)$ for all real $\lambda$ such that $|\lambda|>\omega_0$ and
	$\|(i\lambda-\mathbf L)^{-1}\|_{\mathbf X\mapsto \mathbf X}\lesssim\frac{1}{|\lambda|}$; for Banach space triplet, we need further
	$\|(i\lambda-\mathbf L)^{-1}\|_{\mathbf Y\mapsto \mathbf X}\lesssim\frac{1}{|\lambda|^{\alpha}}$ for some $\alpha\in [0,1)$.
	
	Then there exists a map $\Psi\in C^k(\mathbf Z_c, \mathbf Z_h)$ and a neighborhood 
	$\mathcal O_{v}\times \mathcal O_{\mu}$ of $(0, 0)$ in $\mathbf Z\times \mathbb R^m$ such that
	
	(a) (tangency) $\Psi(0, 0)=0$ and $D_u\Psi(0,0)=0$.
	
	(b) (local flow invariance) the manifold $\mathcal M_0 (\mu)=\{ v_0+\Psi(v_0,\mu); v_0\in Z_c \}$ has the properties:
	(i) $\mathcal M_0 (\mu)$ is locally invariant, i.e., if $v$ is a solution satisfying $v(0)\in \mathcal M_0 (\mu)\cap \mathcal O_{\mu}$ and $v(t)\in\mathcal O_{v}$ for all $t\in [0, T]$, then
	$v(t)\in\mathcal M_0 (\mu)$ for all $t\in [0, T]$; (ii) $\mathcal M_0 (\mu)$ contains the set of bounded solutions staying in $\mathcal O_{v}$ for all $t\in\mathbb{R}^1$, i.e., if $v$ is a solution satisfying
	$v(t)\in\mathcal O_v$ for all $t\in\mathbb R^1$, then $v(0)\in \mathcal M_0 (\mu)$.
	
	(c) (symmetry) moreover, if the vector field is equivariant in the sense that there
	exists an isometry $\mathbf{T}\in \mathbf L(\mathbf X)\cap \mathbf L(\mathbf Z)$
	which commutes with the vector field in the original system,
	$$[\mathbf{T}, \mathbf{L}]=0,\,\,[\mathbf{T}, \mathbf{R}]=0,$$
	then the $\Psi$ commutes
	with $\mathbf{T}$ on $\mathbf Z_c$: $[\Psi, \mathbf T]=0$.
\end{theorem}

\section{Normal form theorem}
Here we give a version of the normal form theorem with symmetry.

\begin{theorem} (Normal form theorem \cite{HI, LY, Yao})\label{normal form theorem}
	Consider a differential equation in $\mathbb{R}^n$ of the form
	$$\frac{dv}{dt}=\mathbf{L}v+\mathbf{R}(v,\mu)$$ and assume that
	
	(1) $\mathbf{L}$ is a linear map in $\mathbb{R}^n$;\\
	
	(2) for some $k\geq 2$, there exist neighborhoods
	$\mathcal{V_v}\subset\mathbb{R}^n$ and
	$\mathcal{V_\mu}\subset\mathbb{R}^m$ of $(0,0)$ such that
	$\mathbf{R}\in C^k(\mathcal{V}_v\times \mathcal{V_{\mu}},
	\mathbb{R}^n)$ and
	$$\mathbf{R}(0,0)=0,\,\, D_u\mathbf{R}(0,0)=0.$$
	Then for any positive integer p, $k>p\geq2$, there exist neighborhoods $\mathcal{V}_1$ and $\mathcal{V}_2$ of $(0,0)$ in
	$\mathbb{R}^n\times \mathbb{R}^m$ such that for any $\mu\in
	\mathcal{V}_2$, there is a polynomial
	$\Pi_{\mu}:\mathbb{R}^n\rightarrow\mathbb{R}^n$ of degree $p$ with
	the following properties.
	
	(i) The coefficients of the monomials of degree $q$ in $\Pi_{\mu}$
	are functions of $\mu$ of class $C^{k-q}$, and
	$$\Pi_0(0)=0,\,\, D_v \Pi_0(0)=0.$$
	
	(ii) For $v\in\mathcal{V}_1$, the polynomial change of variables
	$$v=w+\Pi_{\mu}(w)$$
	transforms the original system into the normal form
	\begin{equation*}
	\frac{dw}{dt}=\mathbf{L}w+\mathcal{N}_{\mu}(w)+\rho(w,\mu),
	\end{equation*} such that
	
	(a) (tangency) For any $\mu\in \mathcal{V}_2$, $\mathcal{N}_{\mu}$ is a
	polynomial $\mathbb{R}^n\rightarrow\mathbb{R}^n$ of degree $p$, with
	coefficients depending on $\mu$, such that the coefficients of the
	monomials of degree $q$ are of class $C^{k-q}$, and
	$$\mathcal{N}_0(0)=0,\,\, D_v\mathcal{N}_0(0)=0.$$
	
	(b) (characteristic condition) The equality
	$$\mathcal{N}_{\mu}(e^{t\mathbf{L}^*}w)=e^{t\mathbf{L}^*}\mathcal{N}_{\mu}(w)$$
	holds for all $(t,\mu)\in\mathbb{R}\times\mathbb{R}^n$ and
	$\mu\in\mathcal{V}_2$.
	
	(c) (smoothness )The map $\rho$ belongs to
	$C^k(\mathcal{V}_1\times\mathcal{V}_2,\mathbb{R}^n)$, and
	$$\rho(w,\mu)=o(|w|^p)$$
	for all $\mu\in \mathcal{V}_2$.
	
	(d) (symmetry) Moreover, if the vector field is equivariant in the sense that there
	exists an isometry $\mathbf{T}:\mathbb{R}^n\rightarrow\mathbb{R}^n$
	which commutes with the vector field in the original system,
	$$[\mathbf{T}, \mathbf{L}]=0,\,\,[\mathbf{T}, \mathbf{R}]=0,$$
	then the polynomials $\Pi_{\mu}$ and $\mathcal{N}_{\mu}$ commute
	with $\mathbf{T}$ for all $\mu\in\mathcal{V}_2$.
\end{theorem}

We remark that both the center manifolds and normal forms are not unique in general.


{\small
        $^\dagger$ Department of Mathematics,
	University of Iowa, Iowa City, IA, 52242, USA. 
	
	Emails: jinghuayao@gmail.com; yaoj@indiana.edu; jinghua-yao@uiowa.edu

	$^\ddagger$ Department of Mathematics,
	Indiana University, Bloomington,
       
        IN, 47408, USA. 
	
	Email: wang264@indiana.edu 
	
	$^\ddagger$ Corresponding author.

\end{document}